\documentclass[12pt]{article}
\usepackage{amssymb,amsmath}
\usepackage{amsthm}
\usepackage{cases}
\usepackage{amsfonts}
\usepackage{cite,color,xcolor}
\usepackage[margin=1 in]{geometry}
\usepackage[colorlinks,citecolor=blue,urlcolor=blue]{hyperref}
\usepackage[utf8]{inputenc}
 \usepackage[pagewise]{lineno}

\newtheorem{theorem}{Theorem}[section]

\newtheorem{lemma}{Lemma}[section]
\newtheorem{proposition}{Proposition}[section]

\newtheorem{definition}{Definition}[section]
\newtheorem{remark}{Remark}[section]

\usepackage{txfonts}

\newcommand{\bal}{\begin{align}}
\newcommand{\bbal}{\begin{align*}}
\newcommand{\beq}{\begin{equation}}
\newcommand{\eeq}{\end{equation}}
\newcommand{\bca}{\begin{cases}}
\newcommand{\eca}{\end{cases}}
\def\div{\mathord{{\rm div}}}
\newcommand{\pa}{\partial}
\newcommand{\fr}{\frac}
\newcommand{\na}{\nabla}
\newcommand{\De}{\Delta}

\newcommand{\cd}{\cdot}
\newcommand{\ep}{\varepsilon}
\newcommand{\dd}{\mathrm{d}}
\newcommand{\B}{\dot{B}}

\newcommand{\LL}{\tilde{L}}
\newcommand{\R}{\mathbb{R}}

\newcommand{\ges}{\gtrsim}

\newcommand{\D}{\mathrm{div}}

\newcommand{\ee}{\vec{e}}

\newcommand{\f}{\left}
\newcommand{\g}{\right}

\begin{document}
\title{Ill-posedness issue on the Oldroyd-B model in the critical Besov spaces}

\author{Jinlu Li$^{1}$, Yanghai Yu$^{2,}$\footnote{E-mail: lijinlu@gnnu.edu.cn; yuyanghai214@sina.com(Corresponding author); mathzwp2010@163.com} and Weipeng Zhu$^{3}$\\
\small $^1$ School of Mathematics and Computer Sciences, Gannan Normal University, Ganzhou 341000, China\\
\small $^2$ School of Mathematics and Statistics, Anhui Normal University, Wuhu 241002, China\\
\small $^3$ School of Mathematics, Foshan University, Foshan, Guangdong 528000, China}

\date{\today}

\maketitle\noindent{\hrulefill}

{\bf Abstract:} It was proved in \cite[J. Funct. Anal., 2020]{AP} that the Cauchy problem for some Oldroyd-B model is well-posed in $\B^{d/p-1}_{p,1}(\R^d) \times \B^{d/p}_{p,1}(\R^d)$ with $1\leq p<2d$. In this paper, we prove that the Cauchy problem for the same Oldroyd-B model is ill-posed in $\B^{d/p-1}_{p,r}(\R^d) \times \B^{d/p}_{p,r}(\R^d)$ with $1\leq p\leq \infty$ and $1< r\leq\infty$ due to the lack of continuous dependence  of the solution.


\vskip0mm\noindent{\hrulefill}

\section{Introduction}
In this paper, we consider the Cauchy problem of the following multidimensional ($d\geq2$) Oldroyd-B model:
\begin{equation}\label{0}
\begin{cases}
\pa_tu+u\cd\na u-\De u+\na \mathrm{p}=\D\ \tau,& (t,x)\in\mathbb{R}^{+} \times \mathbb{R}^{d}, \\
\pa_t\tau+u\cd\na \tau +\tau\omega-\omega \tau=0,& (t,x)\in\mathbb{R}^{+} \times \mathbb{R}^{d}, \\
\D\ u=0,& (t,x)\in\mathbb{R}^{+} \times \mathbb{R}^{d}, \\
(u, \tau)(t=0)=\left(u_0, \tau_0\right), & x\in\mathbb{R}^{d},
\end{cases}
\end{equation}
The unknown function $u=u(t, x)$ is the velocity field of a particle $x \in \mathbb{R}^d$ at a time $t \in \mathbb{R}^+$, $\tau=(\tau(t, x))_{i j}$ with $1\leq i,j\leq d$ is the conformation tensor in $\mathbb{R}^{d\times d}$, which describes the internal elastic forces that the constitutive molecules exert on each other, and $\mathrm{p}(t,x)$ is the pressure. The evolutionary equation for the conformation tensor $\tau$ is then driven by the vorticity tensor $\omega$, which stands for the skew-adjoint part of the deformation tensor $\nabla u$:
$$\omega=\frac{\nabla u-\nabla^{\top} u}{2}.$$

The Oldroyd-B model is a typical prototypical model for viscoelastic flows, which describes the hydrodynamics of some specific viscoelastic fluids. We refer the readers to
\cite{cm,fgo,o} for more detailed physical background and derivations of Oldroyd-B type model. The more general version of Oldroyd-B model (see \cite{cm,e1,o} etc.) reads:

\begin{equation}\label{01}
\begin{cases}
\partial_t u+u\cdot \nabla u-\nu \Delta u+\nabla \mathrm{p}=\mu_1\operatorname{div} \tau, & (t,x)\in\mathbb{R}^{+} \times \mathbb{R}^{d}, \\
\partial_t \tau+u \cdot \nabla \tau+\alpha\tau+\mathbb{Q}(u, \tau)=\mu_2 \mathbb{D}(u), & (t,x)\in\mathbb{R}^{+} \times \mathbb{R}^{d}, \\
\operatorname{div} u=0, & (t,x)\in\mathbb{R}^{+} \times \mathbb{R}^{d}, \\
(u, \tau)(t=0)=\left(u_0, \tau_0\right), & x\in \mathbb{R}^{d}.
\end{cases}
\end{equation}
Here the parameters $\mu_1>0, \mu_2>0$ and $\alpha \geq 0$. Moreover, $\nu$ is the coefficient of viscosity.
The bilinear term $\mathbb{Q}(u, \tau)$ is given by
$$
\mathbb{Q}(u, \tau)=\tau\omega -\omega\tau +\beta\f(\mathbb{D}(u) \tau+\tau \mathbb{D}(u)\g), \quad \beta \in[-1,1],
$$
where $\mathbb{D}(u)=\left(\nabla u+\nabla^{\top} u\right)/2$ is the symmetric contribution of the deformation tensor $\nabla u$.

We should note that \eqref{0} can be seen as a reduced version of the Oldroyd-B model \eqref{01}. Indeed,
when imposing the following restriction on the main parameters of \eqref{01}:
\begin{align}\label{02}
\mu_1=\nu=1,\quad \mu_2=0,\quad \alpha=\beta=0,
\end{align}
we can obtain our main system \eqref{0}. The first condition is introduced just for
the sake of a clear presentation, while the second and third conditions will play a major
role in the analysis techniques. The Oldroyd-B model \eqref{01} with \eqref{02} has been considered in \cite{AP}.

\subsection{A review of related results}
\quad
The modeling and analysis of the Oldroyd-B model has attracted much attention over the last decades because of its physical applications and mathematical significance. We present an overview of some well-posedness results concerning systems \eqref{0} and \eqref{01}.
Guillop\'{e} and Saut \cite{gs} obatained the existence and uniqueness of local strong solutions for system \eqref{01} in Sobolev space $H^s(\Omega)$. They also proved that, for sufficiently-smooth bounded domains $\Omega\subset\R^3$ and a sufficiently large regularity $s>0$, these solutions are global if the initial data as well as the coupling parameter are sufficiently small.
Lions and Masmoudi \cite{lm} addressed the bidimensional system \eqref{01} in a corotational setting $\beta=0$ and showed the existence and uniqueness of global-in-time weak solutions. Constantin and Kliegl \cite{p12} considered the the fully parabolic Oldroyd-B model \eqref{01} with $\alpha>0$ in dimension two and obtained the global well-posedness for strong solutions, where the corotational assumption $\beta=0$ is not needed. In the absence of the diffusive term and the presence of the damping term, Elgindi and Rousset \cite{e1,e2} addressed the well-posedness of a system related to \eqref{01}. By introducing a dissipative and damping mechanism on the evolution of
the deformation tensor, they considered the case of a null viscosity $\nu=0$ and obtained the global existence of classical solution for large initial data in dimension two when neglecting the bilinear term \eqref{02}. Furthermore, for a general bilinear term \eqref{02}, they proved the existence of global-in-time
classical solution for small initial data. An energy variational approach was also introduced by Lin, Liu and Zhang \cite{lilz} to describe the motion of viscoelastic fluids.
Lei, Liu and Zhou \cite{llz} proved existence and uniqueness of classical solutions near
equilibrium of system \eqref{01} for small initial data by assuming the domain to be periodic or
to be the whole space.

Another remarkable way to study the Oldroyd-B model is constructing solutions
in {\it scaling invariant spaces}.
We can check that if $(u, \tau)$ solves \eqref{0}, so does $\left(u_\ell(t, x), \tau_\ell(t, x)\right)$, where$$
\left(u_\ell(t, x), \tau_\ell(t, x)\right)=\left(\ell u(\ell^2 t, \ell x),  \tau(\ell^2 t, \ell x)\right), \quad \ell>0.
$$
This suggests us to choose initial data $(u_0, \tau_0)$ in ``critical spaces" whose norm is invariant for all $\ell>0$ (up to a constant independent of $\ell$) by the transformation $(u_0, \tau_0)(x) \mapsto(\ell u_0(\ell x), \tau_0(\ell x))$.
It is natural that $\dot{B}_{p, r}^{{d}/{p}-1}\times\dot{B}_{p, r}^{{d}/{p}}$ are critical spaces to the Oldroyd-B model \eqref{0}. Chemin and Masmoudi \cite{cm} first established the global well-posedness result of system
\eqref{01} in the critical $L^p$ framework. They particularly showed the local and global-in-time
existence of solutions for large and small initial data, respectively, under the assumption
of a smallness condition on the coupling parameters of system \eqref{01}. Some
improvements were made by Chen and Miao in \cite{chm}. Zi, Fang and Zhang \cite{zfz} removed the smallness
restriction on the coupling parameter. Fang and Zi \cite{z16} also constructed strong solutions to the Oldroyd-B
model \eqref{01} with large vertical initial velocity in the critical $L^2$ framework. For more results on the Oldroyd-B
model, we refer to see \cite{ch,hlz,pwz,wwx,wz,zy} and the references therein.

Anna and Paicu \cite{AP} studied the Oldroyd-B model \eqref{0} under suitable
condition on the initial data and they showed the existence of global-in-time classical solutions in dimension two for large initial data,
uniqueness in dimension two of strong solutions,
and existence and uniqueness of global-in-time strong solution in dimension $d\geq3$ with a Fujita-Kato smallness condition for the initial data.
In particular, they established the following well-posedness result
\begin{theorem}[\cite{AP}]\label{lw} Let $p \in[1,2d)$ and $d\geq3$. Assume that the initial data $u_0$ be a free divergence vector field in $\dot{B}_{p, 1}^{\frac{d}{p}-1}$, while $\tau_0$ belongs to $\dot{B}_{p, 1}^{\frac{d}{p}}$. Then there exists a time $T^*>0$ for which system \eqref{0} admits a unique local solution $(u, \tau)$ satisfying
$$
u \in \mathcal{C}\left([0, T^*), \dot{B}_{p, 1}^{\frac{d}{p}-1}\right) \cap L^1\left(0, T^*, \dot{B}_{p, 1}^{\frac{d}{p}+1}\right), \quad \tau \in \mathcal{C}\left([0, T^*), \dot{B}_{p, 1}^{\frac{d}{p}}\right).
$$
\end{theorem}
From the PDE's point of view, it is crucial to know if an equation which models a physical phenomenon is well-posed in the
Hadamard's sense: existence, uniqueness, and continuous dependence of the solutions with respect to the initial data. In particular, the lack of continuous dependence would cause incorrect solutions or non meaningful solutions. Indeed, this means that the corresponding equation is ill-posed. Many results with regard to the ill-posedness for some important nonlinear PDEs have been obtained. We would like to mention the following incompressible non-resistive magnetohydrodynamics (MHD) equations
\begin{equation}\label{mhd}
\begin{cases}
          \partial_tu+u\cd\na u-\Delta u+\nabla \mathrm{p}=b\cdot\nabla b, & (t,x)\in\mathbb{R}^{+} \times \mathbb{R}^{d}, \\
          \partial_t b + u\cdot\nabla b=b\cdot\nabla u, & (t,x)\in\mathbb{R}^{+} \times \mathbb{R}^{d}, \\
         \div \,u=\div\, b=0, & (t,x)\in\mathbb{R}^{+} \times \mathbb{R}^{d}, \\
         (u,b)(t=0)=(u_0,b_0),& x\in \R^d.
\end{cases}
\end{equation}
Here the initial data $(u_0,b_0)$ is divergence-free. $u=u(t,x)$ represents the velocity
field, $b=b(t,x)$ denotes the magnetic field and $\mathrm{p}=\mathrm{p}(t,x)$ is the scalar pressure, respectively.
Chen, Nie and Ye \cite{CNY} proved that the ill-posedness of the d-dimension $(d\geq2)$ MHD equations \eqref{mhd} in $\B^{d/p-1}_{p,r}(\R^d) \times \B^{d/p}_{p,r}(\R^d)$ with $1\leq p\leq\infty$ and $r>1$ by constructing complex-valued initial data for showing the norm inflation of the magnetic field which comes from the analysis of nonlinear term $b\cdot\nabla u$.
We refer to see more ill-posedness results on the incompressible Navier-Stokes equations \cite{Bou,wang,yon}, the compressible Navier-Stokes equations \cite{C3,CW,IO22} and so on.
In this paper, we are mainly focused on the ill-posedness of the Cauchy problem \eqref{0} in some critical Besov spaces. Precisely speaking, for the case $1<r\leq \infty$, it is still unknown whether the Cauchy problem  \eqref{0} in $\B^{d/p-1}_{p,r}(\R^d) \times \B^{d/p}_{p,r}(\R^d)$ is well-posed or ill-posed. In this paper, we shall answer this question.

\subsection{Main Result}
\quad
The main result of this paper is the following:
\begin{theorem}\label{th3}
Let $d\geq 2$, $1 \leq p \leq \infty$ and $1< r\leq \infty$. The Cauchy problem \eqref{0} is ill-posed in $\B^{d/p-1}_{p,r}(\R^d) \times \B^{d/p}_{p,r}(\R^d)$. More precisely, there exist a sequence of initial data $(u^n_{0},\tau^n_{0})\in \mathcal{S}(\R^d)$, positive time $\left\{T_n\right\}_n$ with $T_n \rightarrow 0$ as $n \rightarrow \infty$, and a sequence of corresponding smooth solutions $\left(u^n, \tau^n\right)$ such that
$$
\lim _{n \rightarrow \infty}\left(\|u^n_{0}\|_{\B^{\frac dp-1}_{p,r}} +\|\tau^n_{0}\|_{\B^{\frac dp}_{p,r}}\right)= 0
$$
and
$$
\left\|\tau^n(T_n,\cdot)\right\|_{\B^{\frac dp}_{p,r}} \geq c_0
$$
for some positive constant $c_0$.
\end{theorem}
\begin{remark}
Theorem \ref{th3} demonstrates that if $d\geq 2$, $1 \leq p \leq \infty$ and $1< r\leq \infty$, there exists a sequence of initial data which converges to zero in
$\B^{ d/p-1}_{p,r}(\R^d) \times \B^{d/p}_{p,r}(\R^d)$ and yields a sequence of solutions to \eqref{0} which does not converge to zero in $\B^{d/p-1}_{p,r}(\R^d) \times \B^{ d/p}_{p,r}(\R^d)$. In other words, \eqref{0} is ill-posed in $\B^{d/p-1}_{p,r}(\R^d) \times \B^{d/p}_{p,r}(\R^d)$ due to the discontinuity of the solution map at zero.
\end{remark}

\begin{remark}
We would like to mention that, Theorem \ref{th3} holds for the case when $p=r=2$ and demonstrates that \eqref{0} is ill-posed in the critical Sobolev spaces $\dot{H}^{d/2-1}(\R^d) \times \dot{H}^{d/2}(\R^d)$.
\end{remark}

\subsection{Main Idea}
\quad
Given a Lipschitz velocity field $u$, we may solve the following ODE to find the flow induced by $u$:
\begin{align}\label{ode}
\quad\begin{cases}
\frac{\dd}{\dd t}\phi(t,x)=u(t,\phi(t,x)),\\
\phi(0,x)=x,
\end{cases}
\end{align}
which is equivalent to the integral form
\bal\label{n}
\phi(t,x)=x+\int^t_0u(s,\phi(s,x))\dd s.
\end{align}
Considering the $\tau$-equation
\begin{align}\label{pde}
\quad\begin{cases}
\pa_t\tau+u\cd\na\tau=\omega\tau-\tau\omega:=P,\\
\tau(0,x)=\tau_0(x),
\end{cases}
\end{align}
then, from \eqref{ode} and \eqref{pde}, we get
\bbal
\tau(t,\phi(t,x))=\tau_0(x)+\int_0^tP(s,\phi(s,x))\dd s.
\end{align*}
We try to extract the worst nonlinear term $P$ from the $\tau$-equation, which leads to the discontinuous solution map of $\tau$ at time $t=0$ in the metric of $\B^{d/p}_{p,r}$.

Setting $P_0:=\omega_0\tau_0-\tau_0\omega_0$, then we has the following decomposition
\bal\label{new}
\tau(t,\phi(t,x))&=\tau_0(x)+tP_0(x)+\int_0^tP(s,\phi(s,x))-P_0(\phi(s,x))\dd s\nonumber\\
&\quad+\int_0^tP_0(\phi(s,x))-P_0(x)\dd s.
\end{align}
Our key argument is that, by constructing suitable initial data $(u_0,\tau_0)$, the term $tP_0(x)$ is the main contribution
to the discontinuity of solution map while the other terms can be small and thus be absorbed.
\begin{remark}
We would like to point out the main difference between our contribution and Chen-Nie-Ye's work \cite{CNY}. Compared with the initial velocity field $u_0$ and initial magnetic field $b_0$ which are complex-valued in \cite{CNY}, both initial data $u_0$ and $\tau_0$ in this work are real-valued and seem to be more physically significant. Also, our construction of $(u_0,\tau_0)$ makes the the analysis and computations more clean and may be enlightening to the ill-posedness for some important nonlinear PDEs. In summary, we proved that the Cauchy problem for the Oldroyd-B model \eqref{0} with real-valued initial data is ill-posed in $\B^{d/p-1}_{p,r}(\R^d) \times \B^{d/p}_{p,r}(\R^d)$ with $1\leq p\leq \infty$ and $1< r\leq\infty$ while it is not clear that the Cauchy problem for the MHD equations \eqref{mhd} with real-valued initial data is ill-posed in the same spaces.
\end{remark}

{\bf Organization of this paper.} In Section \ref{sec2}, we list some notations and  recall some lemmas which will be used in the sequel. In Section \ref{sec3} we present the proof of Theorem \ref{th3}.
\section{Preliminaries}\label{sec2}

{\bf Notation}\quad The metric $\nabla u$ denotes the gradient of $u$ with respect to the $x$ variable, whose $(i,j)$-th component is given by $(\nabla u)_{ij}=\pa_iu_j$ with $1\leq i,j\leq d$. Then the $(i,j)$-th component of the vorticity tensor $\omega=(\nabla u-\nabla^{\top} u)/2$ is given by $(\omega)_{i j}=(\partial_iu_j-\partial_ju_i)/2$.
Throughout this paper, $C$ stands for some positive constant independent of $n$, which may vary from line to line.
The symbol $A\approx B$ means that $C^{-1}B\leq A\leq CB$.
Given a Banach space $X$, we denote its norm by $\|\cdot\|_{X}$.
For $I\subset\R$, we denote by $\mathcal{C}(I;X)$ the set of continuous functions on $I$ with values in $X$.
We use $\mathcal{S}(\R^d)$ and $\mathcal{S}'(\R^d)$ to denote Schwartz functions and the tempered distributions spaces on $\R^d$, respectively.

Next, we will recall some facts about the Littlewood-Paley decomposition and the homogeneous Besov spaces (see \cite{B} for more details). Firstly, let us recall that for all $f\in \mathcal{S}'$, the Fourier transform $\widehat{f}$ is defined by
$$
(\mathcal{F} f)(\xi)=\widehat{f}(\xi)=\int_{\R^d}e^{-\mathrm{i}x\cdot\xi}f(x)\dd x \quad\text{for any}\; \xi\in\R^d.
$$
The inverse Fourier transform allows us to recover $f$ from $\widehat{f}$
$$
f(x)=\mathcal{F}^{-1}\widehat{f}(x)=(2\pi)^{-d}\int_{\R^d}e^{\mathrm{i}x\cdot\xi}\widehat{f}(\xi)\dd\xi.
$$
Choose a radial, non-negative, smooth function $\vartheta:\R\mapsto [0,1]$ such that
 ${\rm{supp}} \,\vartheta\subset B(0, 4/3)$ and $\vartheta(\xi)\equiv1$ for $|\xi|\leq3/4$.
Setting $\varphi(\xi):=\vartheta(\xi/2)-\vartheta(\xi)$, then we deduce that $\varphi$ has the following properties
\begin{itemize}
  \item ${\rm{supp}} \;\varphi\subset \left\{\xi\in \R^d: 3/4\leq|\xi|\leq8/3\right\}$;
  \item $\varphi(\xi)\equiv 1$ for $4/3\leq |\xi|\leq 3/2$;
  \item $\sum_{j\in \mathbb{Z}}\varphi(2^{-j}\xi)=1$ for any $0\neq\xi\in \R^d$.
\end{itemize}
For every $u\in \mathcal{S'}(\mathbb{R}^d)$, the homogeneous dyadic block ${\dot{\Delta}}_j$ is defined as follows
\bbal
&\dot{\Delta}_ju=\varphi(2^{-j}D)u=\mathcal{F}^{-1}\big(\varphi(2^{-j}\cdot)\mathcal{F}u\big)
=2^{dj}\int_{\R^d}\check{\varphi}\big(2^{j}(\cdot-y)\big)u(y)\dd y,\quad \forall j\in \mathbb{Z}.
\end{align*}
In the homogeneous case, the following Littlewood-Paley decomposition makes sense
$$
u=\sum_{j\in \mathbb{Z}}\dot{\Delta}_ju\quad \text{for any}\;u\in \mathcal{S}'_h(\R^d),
$$
where $\mathcal{S}'_h$ is given by
\begin{eqnarray*}
\mathcal{S}'_h:=\Big\{u \in \mathcal{S'}(\mathbb{R}^{d}):\; \lim_{j\rightarrow-\infty}\left\|\vartheta(2^{-j}D)u\right\|_{L^{\infty}}=0 \Big\}.
\end{eqnarray*}
We turn to the definition of the Besov spaces and norms which will come into play in our paper.
\begin{definition}[see \cite{B}]
Let $s\in\mathbb{R}$ and $(p,r)\in[1, \infty]^2$. The homogeneous Besov space $\dot{B}^s_{p,r}(\R^d)$ consists of all tempered distribution $f$ such that
\begin{align*}
\dot{B}_{p,r}^{s}=\Big\{f \in \mathcal{S}'_h(\mathbb{R}^{d}):\; \|f\|_{\dot{B}_{p,r}^{s}(\mathbb{R}^{d})}< \infty \Big\},
\end{align*}
where, for $1\leq r<\infty$,
$$\|f\|_{\dot{B}^{s}_{p,r}(\R^d)}:=\left(\sum_{j\in\mathbb{Z}}2^{sjr}\|\dot{\Delta}_jf\|^r_{L^p(\R^d)}\right)^{1/r}$$
and $\|f\|_{\dot{B}^{s}_{p,\infty}(\R^d)}:=\sup_{j\in\mathbb{Z}}2^{sj}\|\dot{\Delta}_jf\|_{L^p(\R^d)}$.
\end{definition}

For $0<T \leq \infty, s \in \mathbb{R}$ and $1 \leq p, r, \rho \leq \infty$, we set (with the usual convention if $r=\infty$ )
$$
\|f\|_{\tilde{L}_{T}^{\rho}\left(\dot{B}_{p, r}^{s}\right)}:=\left(\sum_{j \in \mathbb{Z}} 2^{j s r}\left\|\dot{\Delta}_{j} f\right\|_{L^{\rho}\left(0, T ; L^{p}\right)}^{r}\right)^{1/r}.
$$
Due  to the Minkowski inequality, one has
$$
\|f\|_{\tilde{L}^\rho_T\left(\dot{B}_{p, r}^s\right)} \leq C\|f\|_{L^\rho_T\left(\dot{B}_{p, r}^s\right)} \quad \text { for } \quad \rho \geq r
$$
while the opposite inequality holds when $\rho \leq r$.

The following Bernstein's inequalities will be used in the sequel.
\begin{lemma}[see \cite{B}] \label{lem2.1} Let $\mathcal{B}$ be a ball and $\mathcal{C}$ be an annulus. There exists a constant $C>0$ such that for all $k\in \mathbb{N}\cup \{0\}$, any positive real number $\lambda$ and any function $f\in L^p$ with $1\leq p \leq q \leq \infty$, we have
\begin{align*}
&{\rm{supp}}\ \widehat{f}\subset \lambda \mathcal{B}\;\Rightarrow\; \|\nabla^kf\|_{L^q}\leq C^{k+1}\lambda^{k+(\frac{d}{p}-\frac{d}{q})}\|f\|_{L^p},  \\
&{\rm{supp}}\ \widehat{f}\subset \lambda \mathcal{C}\;\Rightarrow\; C^{-k-1}\lambda^k\|f\|_{L^p} \leq \|\nabla^kf\|_{L^p} \leq C^{k+1}\lambda^k\|f\|_{L^p}.
\end{align*}
\end{lemma}
As a direct result of Bernstein's inequalities and the definition of Besov Spaces $\dot{B}_{p, r}^s$, we have

\begin{lemma}[see \cite{B}]\label{ps0}
\begin{itemize}
  \item There exists a constant $C>0$ such that
$$
C^{-1}\|f\|_{\dot{B}_{p, r}^s} \leq\|\nabla f\|_{\dot{B}_{p, r}^{s-1}} \leq C\|f\|_{\dot{B}_{p, r}^s} .
$$
  \item If $p \in[1, \infty]$ then $\dot{B}_{p, 1}^{d/ p}(\R^d) \hookrightarrow \dot{B}_{p, \infty}^{d / p} \cap L^{\infty}(\R^d)$. Furthermore, for any $p \in[1, \infty], \dot{B}_{p, 1}^{d/ p}(\R^d)$ is an algebra embedded in $L^{\infty}(\R^d)$.
\end{itemize}
\end{lemma}

\begin{lemma}[see \cite{B}]\label{ps}
Let $s>0,1 \leq p \leq \infty$ and $1 \leq \rho, \rho_{1}, \rho_{2}, \rho_{3}, \rho_{4} \leq \infty$. Then
$$
\|fg\|_{\LL_{T}^{\rho}\left(\dot{B}_{p, r}^{s}\right)} \leq C\left(\|f\|_{L_{T}^{\rho_{1}}\left(L^{\infty}\right)}\|g\|_{\LL_{T}^{\rho_{2}}\left(\dot{B}_{p, r}^{s}\right)}+\|g\|_{L_{T}^{\rho_{3}}\left(L^{\infty}\right)}\|g\|_{\LL_{T}^{\rho_{4}}\left(\dot{B}_{p, r}^{s}\right)}\right),
$$
where
$$
\frac{1}{\rho}=\frac{1}{\rho_{1}}+\frac{1}{\rho_{2}}=\frac{1}{\rho_{3}}+\frac{1}{\rho_{4}}.
$$
\end{lemma}
\begin{lemma}[see \cite{ny1}]\label{ps1}
Let $1\leq \rho,\rho_1,\rho_2\leq \infty$ with $\frac1\rho=\frac1\rho_1+\frac1\rho_2$ and $1\leq p<2d$. Then
\bbal
\|fg\|_{\LL^\rho_T(\B^{\frac dp-1}_{p,1})}\leq C\|f\|_{\LL^{\rho_1}_T(\B^{\frac dp-1}_{p,1})}\|g\|_{\LL^{\rho_2}_T(\B^{\frac dp}_{p,1})}.
\end{align*}
\end{lemma}
Next, we recall the regularity estimates for the transport and heat equations.
\begin{lemma}[see \cite{B}]\label{rg-t} Let $1 \leq p \leq p_1 \leq \infty$ and $s \in\left(-d \min \left\{\frac{1}{p_1}, 1-\frac{1}{p}\right\}, 1+\frac{d}{p_1}\right]$. Let $v$ be a vector field such that $\nabla v \in L_T^1\left(\dot{B}_{p_1, 1}^{d/{p_1}}(\mathbb{R}^d)\right)$. There exists a constant $C$ depending on $p, s, p_1$ such that all solutions $u \in \LL_T^{\infty}\left(\dot{B}_{p, 1}^s(\mathbb{R}^d)\right)$ of the transport equation
\begin{equation*}
\begin{cases}
\partial_t u+v \cdot \nabla u=g\in L_{\text{loc }}^1\left(\mathbb{R}^{+}; \dot{B}_{p, 1}^s(\mathbb{R}^d)\right), \\
u(t=0)=u_0\in \dot{B}_{p, 1}^s(\mathbb{R}^d).
\end{cases}
\end{equation*}
Furthermore, we have, for $t \in[0, T]$,
$$
\|u\|_{\LL_T^{\infty}(\dot{B}_{p, 1}^s)} \leq e^{C V_{p_1}(t)}\left(\left\|u_0\right\|_{\dot{B}_{p, 1}^s}+\int_0^t e^{-C V_{p_1}(\tau)}\|g(\tau)\|_{\dot{B}_{p, 1}^s} \mathrm{~d} \tau\right),
$$
where $V_{p_1}(t)=\int_0^t\|\nabla v\|_{\dot{B}_{p_1, 1}^{p_1}(\mathbb{R}^d)} \mathrm{d} s$.
\end{lemma}

\begin{lemma}[see \cite{B}]\label{rg-h} Let $s \in \mathbb{R}, 1 \leq p, r, \rho_1 \leq \infty$. For some positive time $T$ (possibly $T=\infty$ ), then the heat equation
$$
\begin{cases}
\partial_t u-\nu \Delta u=f\in \tilde{L}^{\rho_1}\left(0, T ; \dot{B}_{p, r}^{s-2+2 / \rho_1}(\mathbb{R}^d)\right),  \\
u(t=0)=u_0\in \dot{B}_{p, r}^s(\mathbb{R}^d),
\end{cases}
$$
admits a unique strong solution in $\tilde{L}^{\infty}_T\left(\dot{B}_{p, r}^s\right) \cap \tilde{L}^{\rho_1}_T\left(\dot{B}_{p, r}^{s-2+2 / \rho_1}\right)$. Moreover, there exists a constant $C$ depending just on the dimension $d$ such that the following estimate holds true for any time $t \in[0, T]$:
$$
\nu^{\frac{1}{\rho}}\|u\|_{\tilde{L}^\rho_t(\dot{B}_{p, r}^{s+\frac{2}{\rho}})} \leq C\left(\left\|u_0\right\|_{\dot{B}_{p, r}^s}+\nu^{\frac{1}{\rho_1}-1}\|f\|_{\tilde{L}^{\rho_1}_t(\dot{B}_{p, r}^{s-2+\frac{2}{\rho_1}})}\right),\quad \rho \geq \rho_1.
$$
\end{lemma}
\begin{remark} Due to Lemma \ref{rg-h}, and using the fact that the projector $\mathcal{P}$ on the free divergence vector fields is continuous from $\dot{B}_{p, r}^s$ to itself since it is a homogeneous Fourier multiplier of degree 0, we can easily solve the non-stationary Stokes problem
$$
\begin{cases}\partial_t u-\nu \Delta u+\nabla \mathrm{p}=f, & {[0, T) \times \mathbb{R}^{d},} \\
\operatorname{div} u=0, & {[0, T) \times \mathbb{R}^{d},} \\
u(t=0)=u_0, & \mathbb{R}^{d}.\end{cases}
$$
Assume that $u_0 \in \dot{B}_{p, r}^s$ and $f\in\tilde{L}^1\left(0, T ; \dot{B}_{p, r}^s\right)$, then there exists a unique solution $(u,\mathrm{p})$ satisfying
$$
u \in \tilde{L}^{\infty}\left(0, T ; \dot{B}_{p, r}^s\right) \cap \tilde{L}^1\left(0, T ; \dot{B}_{p, r}^{s+2}\right), \quad \nabla \mathrm{p} \in \tilde{L}^1\left(0, T ; \dot{B}_{p, r}^s\right).
$$
Furthermore, it holds
$$
\nu^{\frac{1}{\rho}}\|u\|_{\tilde{L}^\rho(0, T ; \dot{B}_{p, r}^{s+\frac{2}{\rho}})} \leq C\left(\left\|u_0\right\|_{\dot{B}_{p, r}^s}+\|\mathcal{P} f\|_{L^1\left(0, T ; \dot{B}_{p, r}^s\right)}\right) .
$$
\end{remark}

\begin{lemma}[see \cite{B}]\label{fh1} Assume that $u$ is a smooth vector field and $\phi(t, x)$ satisfies \eqref{n}.
Then, for all $t \in \mathbb{R}^{+}$, the flow $\phi(t, x)$ is a $C^1$-diffeomorphism over $\mathbb{R}^d$, and we have
\begin{align*}
&\left\|\phi(t,x)-x\right\|_{L_x^{\infty}} \leq\int_0^t\|u(s)\|_{L_x^{\infty}} \mathrm{d} s,\\
&\left\|\nabla\phi^{ \pm}(t)\right\|_{L_x^{\infty}} \leq \exp \left(\int_0^t\|\nabla u(s)\|_{L_x^{\infty}} \mathrm{d} s\right) .
\end{align*}
\end{lemma}
\begin{lemma}[see \cite{CNY}]\label{fh2} Let $s \in(-1,1)$ and $(p, q) \in[1, \infty]^2 $. Let $u \in \mathcal{S}(\mathbb{R}^d)$ with $\operatorname{div} u=0$.  Then the flow $\phi$ which is defined by $u$ in \eqref{n} and its inverse $\phi^{-}$ are $C^1$ measure-preserving global diffeomorphism over $\mathbb{R}^d$. There holds that
$$
\|u \circ \phi\|_{\dot{B}_{p, q}^s} \leq C \exp \left(\int_0^t\|\nabla u(s)\|_{L_x^{\infty}} \mathrm{d} s\right)\|u\|_{\dot{B}_{p, q}^s}.
$$
\end{lemma}

\section{Proof of Theorem \ref{th3}}\label{sec3}
In this section, we prove Theorem \ref{th3} by dividing it into several parts:
(1) Construction of initial data;
(2) Estimation of initial data;
(3) Local well-posedness for \eqref{0} with initial data \eqref{u0-de}-\eqref{t0-de};
(4) Discontinuity of the solution map.
\subsection{Construction of initial data}
\quad
We present here a list of the key ingredients used in our construction of initial data.
\begin{itemize}
  \item Before constructing the sequence of initial data, we need to introduce smooth, radial cut-off functions to localize the frequency region. We define an even, real-valued and non-negative function $\widehat{\theta}\in \mathcal{C}^\infty_0(\mathbb{R})$ with values in $[0,1]$ which satisfies
\bbal
\widehat{\theta}(\xi)=
\bca
1, \quad \mathrm{if} \ |\xi|\leq \frac{1}{200 d},\\
0, \quad \mathrm{if} \ |\xi|\geq \frac{1}{100 d}.
\eca
\end{align*}
We assume that $\widehat{\theta}$ is even and real-valued such that $\theta$ is a real-valued function. Then we can define the new real-valued function
$$\phi(x)=\prod_{i=1}^d\theta(x_i)\quad \text{with}\quad \phi(0)=\theta^d(0)>0.$$
\item We write
  $$\Gamma_n:=\frac{1}{\ln\ln n},\quad 1\ll n\in 2\mathbb{N}=\left\{2,4,6,\cdots\right\}.$$
  \item Let $0<\varepsilon \ll 1$ ($\varepsilon$ will be chosen below). We define the matric $A$ whose $(i,j)$-th component $(A)_{ij}$ with $1 \leq i , j \leq d$ is given by
\bbal
&d=2,\quad\left(A\right)_{ij}:=\bca
\varepsilon, \quad  1\leq i=j\leq d,\\
0, \quad else,
\eca\\
&d\geq3,\quad
\left(A\right)_{ij}:=\bca
\varepsilon, \quad  1\leq i=j\leq 2,\\
1, \quad  3\leq i=j\leq d,\\
0, \quad else.
\eca
\end{align*}
  \item We define the vector
  $$\ee:=\frac{\sqrt{2}}{2}(1,1,\underbrace{0,\cdots,0}_{d-2}).$$
\end{itemize}
With the above, we introduce
\bbal
&b_n:=\Gamma_n2^{n}\phi\left(2^{n}Ax\right)
\sin\left(\frac{17}{12}2^{n}\ee\cdot x\right), \\
&c_n:=\mathcal{F}^{-1}\left(\frac{\xi_2-\xi_1}{\xi_2}\widehat{b}_n\right)=b_n-\mathcal{F}^{-1}L,\quad L(\xi):=\frac{\xi_1}{\xi_2}\widehat{b_n}(\xi),\\
&d_n:=\Gamma_n\sum^{n/2}_{j=1}\frac{1}{j}\tilde{b}_j(x),\quad\tilde{b}_j(x):=\phi(2^{j}x)\cos\left(\frac{17}{12}2^{j}\ee\cdot x\right).
\end{align*}
Obviously,  $b_n$ is a real scalar function. A trivial computation gives that
\bal\label{hyy1}
\widehat{b_n}(\xi)=&\frac{\mathrm{i}\Gamma_n2^{n}}{\varepsilon^{2}2^{dn+1}} \f\{\prod_{i=1}^2\widehat{\theta}\left(\frac{\xi_i+\widetilde{\lambda}_n}{\varepsilon2^{n}}\right)
-\prod_{i=1}^2\widehat{\theta}\left(\frac{\xi_i-\widetilde{\lambda}_n}{\varepsilon2^{n}}\right)\g\}\prod_{i=3}^d\widehat{\theta}\left(\frac{\xi_i}{2^{n}}\right),
\end{align}
where $\widetilde{\lambda}_n=\frac{17\sqrt{2}}{24}2^n$.
Obviously, one has
\bbal
&\widehat{b_n}(-\xi_1,-\xi_2,\xi_3,\cdots,\xi_d)=-\widehat{b_n}(\xi_1,\xi_2,\xi_3,\cdots,\xi_d),
\end{align*}
which implies that
\bbal
&L(-\xi_1,-\xi_2,\xi_3,\cdots,\xi_d)=-L(\xi_1,\xi_2,\xi_3,\cdots,\xi_d).
\end{align*}
Hence, we can deduce that $c_n$ is also a real scalar function and
\bal\label{div}
(\pa_1-\pa_2)b_n=-\pa_2c_n.
\end{align}
From \eqref{hyy1}, we also have
\bal\label{hyy2}
&\mathrm{supp} \ \widehat{b_n}(\xi)\subset \bf{S}_1 \cup \bf{S}_2,
\end{align}
where
\bbal
&{\bf{S}_1}:= \left\{\xi\in\R^d: \frac{17\sqrt{2}}{24}2^{n}-\frac{\varepsilon2^n}{100d}\leq \xi_1,\xi_2\leq \frac{17\sqrt{2}}{24}2^{n}+\frac{\varepsilon2^n}{100d},\ |\xi_i|\leq \frac{2^n}{100d},\ 3\leq i\leq d\right\}, \\
&{\bf{S}_2}:= \left\{\xi\in\R^d: -\frac{17\sqrt{2}}{24}2^{n}-\frac{\varepsilon2^n}{100d}\leq \xi_1,\xi_2\leq -\frac{17\sqrt{2}}{24}2^{n}+\frac{\varepsilon2^n}{100d},\ |\xi_i|\leq \frac{2^n}{100d},\ 3\leq i\leq d\right\}.
\end{align*}
In particular, we should emphasize the following important fact
\bbal
&\mathrm{supp} \ \widehat{b_n}(\xi)\subset  \left\{\xi\in\R^d: \ \frac{33}{24}2^{n}\leq |\xi|\leq \frac{35}{24}2^{n}\right\}, \\
&\mathrm{supp} \ \widehat{\tilde{b}_j}(\xi)\subset  \left\{\xi\in\R^d: \ \frac{33}{24}2^{j}\leq |\xi|\leq \frac{35}{24}2^{j}\right\},\quad j\in[1,n/2],
\end{align*}
which gives that
\bbal
&\mathrm{supp} \ \widehat{d_n}(\xi)\subset  \left\{\xi\in\R^d: \ \frac{33}{12}\leq |\xi|\leq \frac{35}{24}2^{\fr{n}{2}}\right\}.
\end{align*}
{\bf Initial Data.} For $1\leq i,j\leq d$, we construct the initial data $u^n_0$ and $\tau^n_0$ whose components are given by
\begin{align}
&(u^n_0)_{i}\equiv\bca
b_n, \quad  &i=1,\\
c_n-b_n, \quad  &i=2,\\
0, \quad &i\geq 3,
\eca\label{u0-de}
\quad\text{and}\\
&(\tau^n_0)_{ij}\equiv\bca
d_n, \quad  &(i,j)\in\{(1,2),(2,1)\},\\
0, \quad &(i,j)\notin\{(1,2),(2,1)\}.
\eca\label{t0-de}
\end{align}
Obviously, from \eqref{div}, one has $\D\ u^n_0=0$. We would like to emphasize that both initial data $u^n_0$ and $\tau^n_0$ are real-valued Schwarz functions.
Next, we need to verify that for large $n$ enough
\begin{itemize}
  \item initial data $u^n_0$ and $\tau^n_0$  is small in $\B^{{d}/{p}-1}_{p,r}\times \B^{{d}/{p}}_{p,r}$ for any $r>1$;
  \item the main contributor $\omega^n_0\tau^n_0-\tau^n_0\omega^n_0$ is large in $L^\infty$.
\end{itemize}

\subsection{Estimation of initial data}
\begin{proposition}\label{pro1}
Let $u^n_0$ and $\tau^n_0$ be defined by \eqref{u0-de}-\eqref{t0-de}. Then for $(p,r)\in[1,\infty]^2$, there exists a positive constant $C=C(\phi)$ independent of $n$ such that for $k\in\{0,1,2,3,4\}$
\bal\label{u0}
&\|u^n_0\|_{\B^{\frac{d}{p}-1+k}_{p,r}}\leq C\ep^{-2}\Gamma_n2^{kn},\\
&\|\tau^n_0\|_{\B^{\frac{d}{p}+k}_{p,r}}\leq C\Gamma_n\left(\sum_{j=1}^{n/2}\frac{1}{j^{r}}\g)^{1/r}2^{kn/2}.
\end{align}
In particular, it holds that
\bbal
\|\tau^n_0\|_{\B^{\frac{d}{p}}_{p,1}}\leq C\Gamma_n\ln n\quad \text{and}\quad\|\tau^n_0\|_{\B^{\frac{d}{p}}_{p,r}}\leq C\Gamma_n\quad\text{for} \;1<r\leq\infty.
\end{align*}
\end{proposition}
\begin{proof} We assume that $(p,r)\in [1,\infty)^2$ without loss of generality.
 Notice that $\mathcal{F}(\dot{\Delta}_jb_n)=\varphi(2^{-j}\cdot)\widehat{b_n}$ for all $j\in \mathbb{Z}$ and
$
\varphi(2^{-j}\xi)\equiv 1$ in $\left\{\xi\in\R^d: \ \frac{4}{3}2^{j}\leq |\xi|\leq \frac{3}{2}2^{j}\right\},
$
then we have
$
\mathcal{F}(\dot{\Delta}_jb_n)=0$ for  $j\neq n,
$
and thus
$
\dot{\Delta}_jb_n=
b_n$ if $j=n$.
Using the definition of Besov space and the fact that $\phi$ is a Schwartz function, yields
$$2^{-kn}\|u^n_0\|_{\B^{\frac{d}{p}-1+k}_{p,r}}\approx\|u^n_0\|_{\B^{\frac{d}{p}-1}_{p,r}}\approx \Gamma_n 2^{\fr{dn}{p}}\|\phi\left(2^{n}Ax\right)\|_{L^p}\leq C\Gamma_n\ep^{-2}.$$
Notice that
$$\widehat{d_n}(\xi)\subset  \left\{\xi\in\R^d: \ |\xi|\leq \frac{35}{24}2^{n/2}\right\},$$ then by Bernstein's inequality, one has
\bbal\|\tau^n_0\|^r_{\B^{\frac{d}{p}+k}_{p,r}}&\leq  C2^{knr/2}\Gamma_n^r\sum_{j=1}^{n/2}\frac{1}{j^{r}}\f\|2^{\fr{d}{p}j}\phi(2^{j}x)\cos\left(\frac{17}{12}2^{j}\ee\cdot x\right)\g\|^r_{L^p}\leq C 2^{knr/2}\Gamma_n^r\sum_{j=1}^{n/2}\frac{1}{j^{r}}.
\end{align*}
This completes the proof of Proposition \ref{pro1}.
\end{proof}
\begin{proposition}\label{pro2}
Let $u^n_0$ and $\tau^n_0$ be defined by \eqref{u0-de}. If $\ep$ is small enough and $n$ is large enough, then there exists $\tilde{c}>0$ independent of $n$ such that
\bbal
\left\|\dot{\Delta}_n\f(\tau^n_0\omega^n_0-\omega^n_0\tau^n_0\g)\right\|_{L^\infty(\R^d)}=\|\tau^n_0\omega^n_0-\omega^n_0\tau^n_0\|_{L^\infty(\R^d)}\geq \tilde{c}2^{2n}\Gamma^2_n\ln n.
\end{align*}
\end{proposition}
\begin{remark}
We should emphasize that, here and in what follows, the positive constant $C$ whose value may vary from line to line, may depend on $\ep$ and $\phi$ but not $n$. The positive constants $\tilde{c}$ and $\tilde{C}$ whose value may vary from line to line, may depend on $\phi$ but not $n$ and $\ep$.
\end{remark}
\begin{proof}  By direct computations, one has
\bbal
&(\tau^n_0\omega^n_0-\omega^n_0\tau^n_0)_{11}=[\pa_2(u^n_0)_{1}-\pa_1(u^n_0)_{2}]d_n,
\\&(\tau^n_0\omega^n_0-\omega^n_0\tau^n_0)_{22}=[-\pa_2(u^n_0)_{1}+\pa_1(u^n_0)_{2}]d_n,
\\&(\tau^n_0\omega^n_0-\omega^n_0\tau^n_0)_{ij}=0, \quad else.
\end{align*}
Notice that
\bbal
[\pa_2(u^n_0)_{1}-\pa_1(u^n_0)_{2}]d_n=(\pa_1+\pa_2)b_nd_n-\pa_1c_nd_n=:h_n,
\end{align*}
and the support conditions of $\widehat{b_n}$ and $\widehat{d_n}$, then
$$\mathrm{supp} \ \widehat{h_n}(\xi)\subset  \left\{\xi\in\R^d: \ \frac{65}{48}2^{n}\leq |\xi|\leq \frac{71}{48}2^{n}\right\},$$
which implies the first equality.

It is easy to verify that
\bal\label{yz1}
&\|(\pa_1+\pa_2)b_nd_n\|_{L^\infty}\geq \big|[(\pa_1+\pa_2)b_nd_n](0)\big| \geq \tilde{c}\Gamma^2_n2^{2n}\phi^2(0)\sum^{n/2}_{j=1}\frac1j.
\end{align}
Recalling the construction of $(b_n,c_n,d_n)$ and \eqref{hyy1}, we deduce that
\bal\label{yz2}
\|\pa_1c_nd_n\|_{L^\infty}&\leq\|\pa_1c_n\|_{L^\infty}\|d_n\|_{L^\infty}\leq \|\xi_1\widehat{c_n}(\xi)\|_{L^1}\|d_n\|_{L^\infty}\nonumber\\
&\leq \f\|\frac{\xi_1(\xi_1-\xi_2)}{\xi_2}\widehat{b_n}\g\|_{L^1}\|d_n\|_{L^\infty}\nonumber\\
&\leq \tilde{C}\ep 2^n\|\widehat{b_n}\|_{L^1}\|d_n\|_{L^\infty}\leq \tilde{C} \Gamma^2_n \ep2^{2n}\sum^{n/2}_{j=1}\frac1j,
\end{align}
where we have used \eqref{hyy2}. In fact, \eqref{hyy2} implies that $|\xi_1-\xi_2|\leq \tilde{C}\ep2^n$ and $|\xi_1|\approx|\xi_2|\approx2^n$ when $\xi\in \mathrm{supp} \ \widehat{b_n}$. Here we would like to emphasize that the above constants $\tilde{c}$ and $\tilde{C}$ do not depend on the parameter $\ep$.

Combining \eqref{yz1} and \eqref{yz2}, yields
\bbal
\|\tau^n_0\omega^n_0-\omega^n_0\tau^n_0\|_{L^\infty}\geq \f(\tilde{c}-\tilde{C}\ep\g)2^{2n}\Gamma^2_n\sum^{n/2}_{j=1}\frac1j\geq \f(\tilde{c}-\tilde{C}\ep\g)2^{2n}\Gamma^2_n\ln n.
\end{align*}
This completes the proof of Proposition \ref{pro2}.
\end{proof}
\subsection{Local well-posedness for \eqref{0} with initial data \eqref{u0-de}-\eqref{t0-de}}
\quad
From now on, we choose ``certain time" as $T=2^{-2n}(\ln n\Gamma_n^2)^{-1}$.
We begin to establish a locally well-posed result for System \eqref{0} with initial data constructed by \eqref{u0-de}-\eqref{t0-de}.

\begin{proposition}\label{pro3}Let $(u^n_0,\tau^n_0)$ be defined by \eqref{u0-de}-\eqref{t0-de}. Given $ 1\leq q< 2d$, there exist some constant $C_0>1$ which may depend on $\ep$ and $\phi$ but not $n$, and $N_0$ such that for $n>N_0$, System \eqref{0} has a unique local solution $(u, \tau)$ associated with initial data $(u^n_0,\tau^n_0)$ satisfying
$$
\begin{aligned}
& u \in  \mathcal{C}\left([0, T], \dot{B}_{q, 1}^{\frac{d}{q}-1} \cap \dot{B}_{q, 1}^{\frac{d}{q}+3}\right) \cap {L}^1\left([0, T], \dot{B}_{q, 1}^{\frac{d}{q}+1} \cap \dot{B}_{q, 1}^{\frac{d}{q}+5}\right), \\
& \tau \in  \mathcal{C}\left([0, T], \dot{B}_{q, 1}^{\frac{d}{q}} \cap \dot{B}_{q, 1}^{\frac{d}{q}+4}\right),
\end{aligned}
$$
and the following estimates hold for $k\in\{0,1,2,3,4\}$
\begin{align*}
&\|u\|_{L^\infty_T(\B^{\frac d{q}-1+k}_{q,1})}+\|u\|_{L^1_T(\B^{\frac d{q}+1+k}_{q,1})}\leq C_0\Gamma_n2^{kn}, \\
&\|\tau\|_{L^\infty_T(\B^{\frac d{q}+k}_{{q},1})}\leq  C_0\Gamma_n\ln n2^{kn}.
\end{align*}
\end{proposition}
\begin{proof}  Since the initial data $(u^n_0,\tau^n_0)$ is in the Schwartz class, we can deduce that $(u,\tau)$ belongs to the smoother class. More precisely, we know from Theorem \ref{lw} that for short time $T$, System \eqref{0} has a unique local solution $(u, \tau)$ satisfying
\bbal
& u \in  \mathcal{C}\left([0, T], \dot{B}_{q, 1}^{\frac{d}{q}-1}\right) \cap {L}^1\left([0, T], \dot{B}_{q, 1}^{\frac{d}{q}+1}\right), \quad
 \tau \in  \mathcal{C}\left([0, T], \dot{B}_{q, 1}^{\frac{d}{q}}\right).
\end{align*}
In fact, using Lemma \ref{rg-h} and Lemma \ref{ps}, one has
\bal\label{u1}
&\|u\|_{\LL^\infty_{T}(\B^{\frac d{q}-1}_{{q},1})}+\|u\|_{L^1_{T}(\B^{\frac d{q}+1}_{q,1})}\leq C\f(\|u^n_0\|_{\B^{\frac d{q}-1}_{q,1}}+\|u\|_{\LL^\infty_{T}(\B^{\frac d{q}-1}_{q,1})}\|u\|_{L^1_{T}(\B^{\frac d{q}+1}_{{q},1})}+T\|\tau\|_{\LL^\infty_{T}(\B^{\frac d{q}}_{{q},1})}\g).
\end{align}
Using Lemma \ref{rg-t} and Lemma \ref{ps1}, one has
\bal\label{t1}
\|\tau\|_{\LL^\infty_{T}(\B^{\frac d{q}}_{{q},1})}\leq \exp \left(C\|u\|_{L_T^1(\dot{B}_{{q}, 1}^{\frac{d}{{q}}+1})}\right)\f(\|\tau^n_0\|_{\B^{\frac d{q}}_{{q},1}}+C\|\tau\|_{\LL^\infty_{T}(\B^{\frac d{q}}_{{q},1})}\|u\|_{L^1_{T}(\B^{\frac d{q}+1}_{{q},1})}\g).
\end{align}
For the sake of convenience, we denote
\bbal
&X_T:=\|u(t,\cdot)\|_{\LL^\infty_{T}(\B^{\frac d{q}-1}_{{q},1})}+\|u(t,\cdot)\|_{L^1_{T}(\B^{\frac d{q}+1}_{{q},1})},\quad
Y_T:=\|\tau(t,\cdot)\|_{\LL^\infty_{T}(\B^{\frac d{q}}_{{q},1})}\quad\text{and}\quad Z_T:=X_T+2CTY_T.
\end{align*}
From \eqref{u1} and \eqref{t1}, it follows that
\bal\label{ul}
Z_T&\leq \f(C+Ce^{CX_T}\g)\f(\|u^n_0\|_{\B^{\frac d{q}-1}_{{q},1}}+T\|\tau^n_0\|_{\B^{\frac d{q}}_{{q},1}}+(Z_T)^2\g)
\leq Ce^{CZ_T}\f(\Gamma_n+(Z_T)^2\g).
\end{align}
By using the continuity argument, we can deduce that for $n$ large enough
\bal\label{u2}
\|u\|_{\LL^\infty_{T}(\B^{\frac d{q}-1}_{{q},1})}+\|u\|_{L^1_{T}(\B^{\frac d{q}+1}_{{q},1})}\leq Z_T\leq 2C\Gamma_n.
\end{align}
Inserting the above into \eqref{t1} yields that
\bal\label{t2}
\|\tau\|_{\LL^\infty_{T}(\B^{\frac d{q}}_{{q},1})}\leq 2\|\tau^n_0\|_{\B^{\frac d{q}}_{{q},1}}\leq C\Gamma_n\ln n.
\end{align}
{\bf Case k=1.} Applying $\nabla$ to Eq. $\eqref{0}_2$, then taking advantage of Lemmas \ref{rg-t}-\ref{ps1} again and using \eqref{u2}-\eqref{t2}, we can infer that
\bbal
\|\tau\|_{\LL_T^{\infty}(\dot{B}_{{q}, 1}^{\frac{d}{q}+1})} &\leq  \exp \left(C\|u\|_{L_T^1(\dot{B}_{{q}, 1}^{\frac{d}{{q}}+1})}\right)\left(\left\|\tau^n_0\right\|_{\dot{B}_{{q}, 1}^{\frac{d}{q}+1}}
+C\|\tau\|_{\LL_T^{\infty}(\dot{B}_{q, 1}^{\frac{d}{q}})}\|u\|_{L_T^1(\dot{B}_{{q}, 1}^{\frac{d}{{q}}+2})}\right.\\
&\quad\left.+C\|\tau\|_{\LL_T^{\infty}(\dot{B}_{q, 1}^{\frac{d}{q}+1})}\|u\|_{L_T^1(\dot{B}_{q, 1}^{\frac{d}{q}+1})}\right)\nonumber\\
&\leq C\Gamma_n\left(2^n\ln n
+\ln n\|u\|_{L_T^1(\dot{B}_{{q}, 1}^{\frac{d}{{q}}+2})}+\|\tau\|_{\LL_T^{\infty}(\dot{B}_{q, 1}^{\frac{d}{q}+1})}\right),
\end{align*}
which implies that for large $n$ enough
\bal\label{t3-0}
\|\tau\|_{\LL_T^{\infty}(\dot{B}_{{q}, 1}^{\frac{d}{q}+1})}
&\leq C\Gamma_n\ln n\left(2^n
+\|u\|_{L_T^1(\dot{B}_{{q}, 1}^{\frac{d}{{q}}+2})}\right).
\end{align}
With the aid of Lemmas \ref{rg-h}-\ref{ps1} again, one can obtain from \eqref{t3-0} that
\bbal
\|u\|_{\LL^\infty_{T}(\B^{\frac d{q}}_{{q},1})}& +\|u\|_{L^1_{T}(\B^{\frac d{q}+2}_{{q},1})}\leq \left\|u^n_0\right\|_{\dot{B}_{q, 1}^{\frac{d}{q}}}+C\|u\|_{\LL_T^{\infty}(\dot{B}_{q, 1}^{\frac{d}{q}})}\|u\|_{L_T^1(\dot{B}_{q, 1}^{\frac{d}{q}+1})}+T\|\tau\|_{\LL_T^{\infty}(\dot{B}_{q, 1}^{\frac{d}{q}+1})}\\
&\leq C\left\{\Gamma_n2^n+\Gamma_n^{-1}2^{-n}+(\Gamma_n+\Gamma_n^{-1}2^{-2n})\f(\|u\|_{\LL^\infty_{T}(\B^{\frac d{q}}_{{q},1})}+\|u\|_{L_T^1(\dot{B}_{{q}, 1}^{\frac{d}{{q}}+2})}\g)\right\},
\end{align*}
which implies that
\bbal
\|u\|_{\LL^\infty_{T}(\B^{\frac d{q}}_{{q},1})}+\|u\|_{L^1_{T}(\B^{\frac d{q}+2}_{{q},1})}
&\leq C\Gamma_n2^n
\end{align*}
and in turn
\bbal
\|\tau\|_{\LL_T^{\infty}(\dot{B}_{{q}, 1}^{\frac{d}{q}+1})}
&\leq C\Gamma_n\ln n2^n.
\end{align*}
{\bf Case k=2,3,4.} Following the similar procedure, we can obtain all the estimations in the case $k=2,3,4.$ We omit the details.
This completes the proof of Proposition \ref{pro3}.
\end{proof}

\subsection{Discontinuity of the solution map}
\quad
Now, we give the lower bound estimation of $\|\tau(T)\|_{\dot{B}^{\frac dp}_{p,r}}$ with $T=2^{-2n}(\ln n\Gamma_n^2)^{-1}$ which is crucial for the proof of the discontinuity of the solution map.

By Lemma \ref{fh2} and using \eqref{new}, one has
\bbal
\|\tau(T)\|_{\dot{B}^{\frac dp}_{p,r}}&\ges\|\tau(T)\|_{\dot{B}^{0}_{\infty,r}}\ges\|\tau(T,\phi(T,x))\|_{\dot{B}^{0}_{\infty,r}}\ges\sup_{j\in \mathbb{Z}}\|\dot{\De}_j\f(\tau(T,\phi(T,x))\g)\|_{L^\infty}
\\& \ges T\sup_{j\in \mathbb{Z}}\|\dot{\De}_jP_0\|_{L^\infty}-\|\tau^n_0\|_{\dot{B}^{\frac dp}_{p,r}}-\int^T_0\|P(s,\phi(s,x))-P_0(\phi(s,x))\|_{L^\infty}\dd s\\
&\quad-\int^T_0\|P_0(\phi(s,x))-P_0(x)\|_{L^\infty}\dd s\\
& \ges T\|\dot{\De}_nP_0\|_{L^\infty}-\|\tau^n_0\|_{\dot{B}^{\frac dp}_{p,r}}-\underbrace{\int_0^T\|P(s,x)-P_0(x)\|_{L^\infty}\dd s}_{=:\mathbf{I}_1}\\
&\quad-\underbrace{\int^T_0\|P_0(\phi(s,x))-P_0(x)\|_{L^\infty}\dd s}_{=:\mathbf{I}_2}.
\end{align*}
Next, we have to present the upper estimations of $\mathbf{I}_1$ and $\mathbf{I}_2$.

{\bf Upper estimation of $\mathbf{I}_1$.}
Notice that for $t\leq T=2^{-2n}(\ln n\Gamma_n^2)^{-1}$, using Lemmas \ref{ps}-\ref{ps0} and Proposition \ref{pro3}, we get
\bbal
\|\tau(t)-\tau^n_0\|_{\B^{d}_{{1},1}}&\leq Ct\|\tau\|_{\LL_t^{\infty}(\B^{d}_{{1},1})}\|u\|_{\LL_t^{\infty}(\B^{d+1}_{{1},1})}+Ct\|\tau\|_{\LL_t^{\infty}(\B^{d+1}_{{1},1})}\|u\|_{\LL_t^{\infty}(\B^{d}_{{1},1})}
\\&\leq Ct\Gamma_n\ln n\Gamma_n2^{2n}+Ct\Gamma_n\ln n2^n\Gamma_n2^n\leq C,
\end{align*}
and
\bbal
\|\omega(t)-\omega^n_0\|_{\B^{d}_{{1},1}}&\leq C\|u(t)-u^n_0\|_{\B^{d+1}_{{1},1}}\\
&\leq C\|u\|_{\LL_t^{\infty}(\B^{d}_{{1},1})}\|u\|_{L_t^{1}(\B^{d+2}_{{1},1})}
+Ct\|u\|_{\LL_t^{\infty}(\B^{d+3}_{1,1})}+Ct\|\tau\|_{\LL_t^{\infty}(\B^{d+2}_{{1},1})}
\\&\leq C\Gamma^2_n2^{2n}+Ct\Gamma_n2^{4n}+Ct\Gamma_n\ln n2^{2n}\\
&\leq C\Gamma^2_n2^{2n}+C(\ln n\Gamma_n)^{-1}2^{2n},
\end{align*}
which in turn give that for $t\leq T=2^{-2n}(\ln n\Gamma_n^2)^{-1}$
\bbal
\|P(t)-P_0\|_{L^\infty}&\leq C\|\omega(t)-\omega^n_0\|_{\B^{d}_{{1},1}}\|\tau(t)\|_{\LL_t^{\infty}(\dot{B}_{1, 1}^{d})}+C\|\tau(t)-\tau^n_0\|_{\B^{d}_{1,1}}\|u^n_0\|_{\LL_t^{\infty}(\dot{B}_{1, 1}^{d+1})}\\
&\leq C\Gamma^3_n\ln n2^{2n}+C2^{2n}.
\end{align*}
Thus, for $t\leq T=2^{-2n}(\ln n\Gamma_n^2)^{-1}$, one has
\bal\label{lyz1}
\mathbf{I}_1
&\leq C\Gamma_n+C(\ln n\Gamma_n^2)^{-1}.
\end{align}
{\bf Upper estimation of $\mathbf{I}_2$.} Using the Newton-Leibniz formula, we have
$$P_0(\phi(s,x))-P_0(x)=\int_0^1(\phi(s,x)-x)\cdot(\nabla P_0)(\theta \phi(s,x)+(1-\theta)x)\dd\theta,$$
from which, we deduce that for $t\leq T=2^{-2n}(\ln n\Gamma_n^2)^{-1}$
\bal\label{lyz2}
\mathbf{I}_2\leq&~\int_0^T\|\nabla P_0\|_{L^\infty}\|\phi(s,x)-x\|_{L^\infty}\dd s\nonumber\\
\leq&~CT^2\f(\|\na \tau^n_0\|_{L^\infty}\|\na u^n_0\|_{L^\infty}+\|\tau^n_0\|_{L^\infty}\|\na^2 u^n_0\|_{L^\infty}\g)\|u\|_{L^\infty_T(L^\infty)}\nonumber\\
\leq&~CT^2\f(\|\tau^n_0\|_{\B^{\frac dp+1}_{p,1}}\|u^n_0\|_{\B^{\frac dp+1}_{p,1}}
+\|\tau^n_0\|_{\B^{\frac dp}_{p,1}}\|u^n_0\|_{\B^{\frac dp+2}_{p,1}}\g)\|u\|_{\LL_T^{\infty}(\dot{B}_{1, 1}^{d})}\nonumber\\
\leq&~C(\ln n\Gamma_n)^{-1},
\end{align}
where we have used Lemma \ref{fh1}.

Recalling that $T=2^{-2n}(\ln n\Gamma_n^2)^{-1}$ and using Propositions \ref{pro1}-\ref{pro2}, one has
\bal\label{lyz3}
T\|\dot{\Delta}_nP_0\|_{L^\infty}\geq \tilde{c}T2^{2n}\ln n\Gamma_n^2=\tilde{c}\quad \text{and}\quad\|\tau^n_0\|_{\B^{\frac{d}{p}}_{p,r}}\leq C\Gamma_n.
\end{align}
Inserting the above \eqref{lyz1}-\eqref{lyz3}, then we get that for large $n$
\bbal
\|\tau(T)\|_{\dot{B}^{\frac dp}_{p,r}}\geq \tilde{c}-C\Gamma_n-C(\ln n\Gamma_n^2)^{-1}-C(\ln n\Gamma_n)^{-1}\geq \tilde{c}/2.
\end{align*}
Meanwhile, we deduce from Proposition \ref{pro3} that
$$\|u\|_{L^\infty_T(\B^{\frac d{p}-1}_{p,1})}+\|u\|_{L^1_T(\B^{\frac d{p}+1}_{p,1})}\leq C_0\Gamma_n\to0\quad\text{as}\quad n\to\infty$$
and Proposition \ref{pro1} that
$$\|u^n_0\|_{\B^{\frac{d}{p}-1}_{p,r}}+\|\tau^n_0\|_{\B^{\frac{d}{p}}_{p,r}}\leq C\Gamma_n\to0\quad\text{as}\quad n\to\infty.$$
Thus, we have obtained a sequence of initial data such that it verifies the discontinuity of solution map. The proof of Theorem \ref{th3} is finished.{\hfill $\square$}

\section*{Acknowledgments}
The authors would like to thank the anonymous referees for valuable comments and suggestions which greatly improved the presentation of this paper. J. Li is supported by the National Natural Science Foundation of China (No.12161004), Innovative High end Talent Project in Ganpo Talent Program (No.gpyc20240069), Training Program for Academic and Technical Leaders of Major Disciplines in Ganpo Juncai Support Program (No.20232BCJ23009). Y. Yu is supported by the National Natural Science Foundation of China (No.12101011). W. Zhu is supported by the National Natural Science Foundation of China (12201118), Guangdong Basic and Applied Basic Research Foundation (No.2021A1515111018) and Guangdong Natural Science Foundation (No.2023A1515010706).
\vspace*{-0.5em}
\section*{Declarations}
\noindent\textbf{Data Availability} No data was used for the research described in the article.

\noindent\textbf{Conflict of interest}
The authors declare that they have no conflict of interest.

\end{document}